\documentclass[11pt]{amsart}
\usepackage[margin=1in]{geometry}
\usepackage[utf8]{inputenc}
\usepackage{amssymb}
\usepackage{amsthm}
\usepackage{tikz}
\usepackage{blindtext}
\usepackage{afterpage}
\usepackage{xcolor}
\usepackage{anyfontsize}
\usepackage{enumitem}
\usepackage{mathrsfs}
\usepackage{amsmath}
\usepackage{mathtools}
\usepackage{color}   
\usepackage{hyperref}
\usepackage{theoremref}
\hypersetup{
    linktoc=all,         
}
\usepackage{tikz-cd}
\usepackage{float}

\newtheorem{theorem}{Theorem}[section]

\newtheorem{lemma}[theorem]{Lemma}

\theoremstyle{definition}

\theoremstyle{theorem}
\newtheorem{proposition}{Proposition}[section]

\theoremstyle{remark}
\newtheorem*{remark}{Remark}

\theoremstyle{definition}

\newcommand{\R}{\mathbb{R}}
\newcommand{\Q}{\mathbb{Q}}
\newcommand{\Z}{\mathbb{Z}}
\newcommand{\C}{\mathbb{C}}
\newcommand{\A}{\mathbb{A}}
\newcommand{\E}{\mathbb{E}}

\newcommand{\ocal}{\mathcal{O}}

\newcommand{\id}[1]{\mathfrak{#1}}

\newcommand{\Real}{\text{Re}}

\usepackage{layout}
\usepackage{pdfpages}

\begin{document}

\title{Probability Laws Concerning Zeta Integrals}
\date{}
\author{Grayson Plumpton} 
\thanks{University of Toronto, \texttt{grayson.plumpton@mail.utoronto.ca}}

\maketitle

\begin{abstract}
\noindent We give a probabilistic interpretation of the Dedekind zeta functions of $\Q(\sqrt{-1})$ and $\Q(\sqrt{-2})$ 
using zeta integrals and use this to show that the first two Li coefficients of these zeta functions are positive. 
This extends a result of Biane, Pitman, and Yor (2001) which considered the case of the Riemann zeta function. 
\end{abstract}

\section{Preliminaries}

\subsection{Introduction}

\noindent A central tool when studying the Riemann zeta function  
is its representation as a Mellin transform. 
Recall that the Mellin transform of a Schwartz function $f$ is the integral transform $\int_0^\infty t^sf(t)\frac{dt}{t}$, where $s \in \C$. 
This integral also has a probabilistic interpretation: if $f(t)\frac{dt}{t}$ is a probability measure on $(0, \infty)$ induced by the random 
variable $Y$, then the Mellin transform of $f$ is $\E(Y^s)$. This is the central idea behind \cite{laws}, in which the Riemann 
zeta function is studied by expressing $\xi(s) = s(s-1)\pi^{-s/2}\Gamma(s/2)\zeta(s)$ 
as a Mellin transform, then interpreting it as being the $s$th moment of a random variable. In this paper, we use 
zeta integrals, essentially Mellin transforms over the locally compact abelian group $\A^\times$, to extend this construction to the 
fields $\Q(\sqrt{-1})$ and $\Q(\sqrt{-2})$. Our main result is the following theorem:

\begin{theorem}\thlabel{rv}
Let $K$ be one of the fields $\Q$, $\Q(\sqrt{-1})$, or $\Q(\sqrt{-2})$. Then there is 
a random variable $X$ such that, for every $s \in \C$, $\E(X^s) = |D|^{-s/2}\xi_K(s)$, where $D$ is the discriminant of $K$ 
and $\xi_K$ is the Dedekind xi function.
\end{theorem}

Using our random variable $X$, we apply some basic properties of cumulants to prove that the first two Li coefficients, introduced 
in \cite{li}, are 
positive for the fields $\Q(\sqrt{-1})$ and $\Q(\sqrt{-2})$. A remarkable result of \cite{li} is that the positivity of all Li coefficients of a Dedekind zeta function  
implies that all of its nontrivial zeroes lie on the critical line.

\subsection{Defining the Dedekind zeta and xi functions (the classical perspective)}

\noindent The following is well known background, details can be found in \cite{li} or \cite{neukirch}. 
Let $K$ be a number field with $r_1$ real and $r_2$ complex embeddings. 
For $\Real(s) > 1$, define the Dedekind zeta function $\zeta_K$ as 
$$\zeta_K(s) = \prod_{\id{p}} \frac{1}{1-N(\id{p})^{-s}}, $$
 the product being taken over all prime ideals of $\ocal_K$, and $N(\id{p}) := \#\ocal_K/\id{p}$. We 
define the real and complex Gamma factors as 
\begin{align*}
	\Gamma_\R(s) &= \pi^{-s/2}\Gamma(s/2), \\
	\Gamma_\C(s) &= (2\pi)^{1-s}\Gamma(s),
\end{align*}
so that the function $Z_K(s) = \Gamma_\R(s)^{r_1}\Gamma_\C(s)^{r_2}\zeta_K(s)$ is analytic on $\C$ except for simple poles 
at $s = 0$ and $s = 1$, and satisfies the functional equation 
$$Z_K(s) = |D|^{1/2-s}Z_K(1-s)$$
where $D$ is the discriminant of $K$. The residues of $Z_K$ at $s = 0$ and $s = 1$ are 
$-c_K$ and $|D|^{-1/2}c_K$ respectively, where 
$$c_K = \frac{2^{r_1}(2\pi)^{r_2}hR}{w}, $$
with $h$ being the class number of $K$, $R$ the regulator, and $w$ the number of roots of unity in $K$. We then define 
$$\xi_K(s) = c_K^{-1}s(s-1)|D|^{s/2}Z_K(s)$$
so that $\xi_K$ is entire, satisfies the symmetric functional equation $\xi_K(s) = \xi_K(1-s)$, and $\xi_K(0) = \xi_K(1) = 1$.

\subsection{The local zeta integral}

\noindent We will now try to understand the function $Z_K$ from the perspective of Fourier analysis on locally compact abelian groups. 
This is of course a rich theory which we cannot do justice; what is shown here 
is primarily meant to clarify notation or any conventions we may use, but for a more complete exposition see 
almost any introductory text on algebraic number theory such as \cite{neukirch}, or \cite{tate}. We begin with the local 
perspective. 

Absolute values on $K$ are of two types. First are those induced by prime ideals $\id{p} \subset \ocal_K$, and 
are called archimedean, finite, or discrete. Second are those induced by embeddings of $K$ into $\R$ or $\C$, 
and are called nonarchimedean or infinite (in particular, there are $r_1 + r_2$ archimedean absolute values on $K$). A 
theorem of Ostrowski says these are (up to equivalence) all the absolute values on $K$ \cite{neukirch}. 
We will denote 
these prime ideals and embeddings in the general by $v$ and call them the places of $K$,  
denoting the absolute value induced by such a place by $|\cdot|_v$, and the completion 
of $K$ with respect to $|\cdot|_v$ by $K_v$ (hence, if $v$ is a real or complex embedding then $K_v$ is $\R$ or $\C$ 
respectively, and otherwise $K_v$ will be a $\id{p}$-adic field for some prime ideal $\id{p}$). For $v$ nonarchimedean, 
the valuation ring of $K_v$ is denoted by $\ocal_v$. 

Each $K_v$ is a local field and hence locally compact, so that its additive group is a locally compact abelian group 
and we can define on it a Haar measure $dx_v$, unique up to multiplication by a constant. To ensure $dx_v$ is self-dual, 
define it to be the usual Lebesgue measure if $v$ is real archimedean, twice the usual Lebesgue measure 
if $v$ is complex archimedean, and the unique Haar measure assigning $\ocal_v$ measure $N(\id{d}_v)^{-1/2}$ if $v$ 
is nonarchimedean, where $\id{d}_v$ is the different ideal of $K_v$. When there is no ambiguity we may denote 
$dx_v$ simply by $dx$. 

A function $f:K_v \to \C$ is called a Schwartz-Bruhat function if either $v$ is archimedean and $f$ is 
a Schwartz function, or $v$ is nonarchimedean and $f$ is a locally constant function of compact support. The 
$\C$-vector space of Schwartz-Bruhat functions is denoted $\mathscr{S}(K_v)$. 

We can also define a multiplicative Haar measure on the group \(K_v^\times\) by letting 
\begin{align*}
d^\times x_v &= \dfrac{dx_v}{|x|_v}, &&\text{if \(v\) is archimedean} \\
d^\times x_\id{p} &= \dfrac{N(\id{p})}{N(\id{p})-1}\dfrac{dx_\id{p}}{|x|_\id{p}}, &&\text{if \(v = \id{p}\) is nonarchimedean}.
\end{align*}

For $f \in \mathscr{S}(K_v)$, we can now define the local zeta integral for $\Real(s) > 1$ as 
$$Z_v(f, s) := \int_{K_v^\times}|x|_v^sf(x)d^\times x_v.$$
(In full generality the zeta integral will also have a factor $\eta(x)$ where $\eta$ is a unitary character of 
the group $K_v^\times$, but for our application we will simply take $\eta = 1$.)

\subsection{The global zeta integral}

Denote by $\A$ the ring of adeles and by $\A^\times$ the group of ideles of $K$. Functions $f: \A \to \C$ of the form 
$f = \bigotimes_v f_v$ with $f_v \in \mathscr{S}(K_v)$ for every $v$ and $f_\id{p} = 1_{\ocal_\id{p}}$ for all but finitely many 
nonarchimedean places $\id{p}$ form a basis for $\mathscr{S}(\A)$ as a $\C$-vector space. We denote by $dx$ the self-dual Haar 
measure on $\A$ formed with the self-dual $dx_v$, so that integrating functions of the form $f = \bigotimes f_v$ factors as 
$$\int_\A f(x) dx = \prod_v \int_{K_v}f_v(x_v)dx_v.$$
The same can be done for the group $\A^\times$, for which the Haar measure will be denoted $d^\times x$. A complete account of the construction of these measures can be found in \cite{tate}.

In general, global constructions on $\A$ and $\A^\times$, such as absolute values, additive characters, the Fourier transform, 
and zeta integrals factor through local ones. That is, for $x = (x_v)_v \in \A^\times$ we define the idelic norm 
as $|x| = \prod_v |x_v|_v$, additive characters on $\A$ are generated by characters of the form $\psi = (\psi_v)_v$, 
so that for $f \in \mathscr{S}(\A)$ the integral of the Fourier transform factors and we have 
$\widehat{f} = \bigotimes_v \widehat{f}_v$, and similarly the global zeta integral defined for $f = \bigotimes_v f_v$ 
as 
$$Z(f, s) = \int_{\A^\times}|x|^s f(x) d^\times x$$
factors as $Z(f, s) = \prod_v Z_v(f_v, s)$. In this case, the main result of Tate's thesis is the functional equation 
$$Z(f, s) = Z(\widehat{f}, 1-s).$$

An important step in proving this is noticing that the idelic norm is trivial on the diagonal $K^\times \subset \A^\times$ (Artin's product formula), 
so that when calculating the zeta integral of $f$ we can make the following change: 
$$\int_{\A^\times} |x|^sf(x) d^\times x = \int_{\A^\times/K^\times}|x|^s\left(\sum_{\kappa \in K^\times}f(x\kappa)\right)d^\times x.$$
This trick will also be of use to us when constructing our density function. It is not immediately obvious that our 
sum over $K^\times$ will converge; however, the nonarchimedean part of $f$ forces the summands to be zero 
outside of a fractional ideal of $K$, so that one is ultimately only summing the archimedean part of $f$ over a 
lattice in $K\otimes \R \simeq \R^{r_1}\times \C^{r_2}$. 

\section{Constructing the random variable}

We previously constructed the function $Z_K$ as being a product over the finite prime ideals of $\ocal_K$, 
and one gamma factor for each archimedean place of $K$; in other words, a product over each of the places 
of $K$. This is analogous to each of our global constructions arising as products over the places of $K$. 
We will now see that this is not a coincidence, and that the function $Z_K$ arises as, in a sense, the simplest zeta integral on 
$\A^\times$. 

\subsection{The functions $Z_K$ and $s(s-1)Z_K(s)$ as zeta integrals}

For each finite place $\id{p}$ of $K$, let 
$$g_\id{p} = N(\id{d}_\id{p})^{1/2}1_{\ocal_\id{p}}$$
(the constant cancels out the measure of $\ocal_\id{p}$). Then for each real embedding \(\sigma\), let 
\[g_{\sigma}(x) = e^{-\pi x^2}\]
and for each complex embedding \(\tau\), let 
\[g_{\tau}(z) = e^{-2\pi z\overline{z}} = e^{-2 \pi |z|^2} .\]

It is straightforward to check that for each finite prime \(\id{p}\), real embedding $\sigma$, and complex embedding $\tau$ that 
\begin{align*}
Z_\id{p}(f_\id{p}, s) &= \frac{1}{1 - N(\id{p})^{-s}},\\
Z_{\sigma}(g_{\sigma}, s) &= \Gamma_\R(s),\\
Z_{\tau}(g_{\tau}, s) &= \Gamma_\C(s).
\end{align*}
Hence, with each \(g_v\) defined above,  \(d^\times x\) the Haar measure on \(\A^\times\), 
and \(g = \bigotimes_v g_v\), we get 
\begin{align*}
Z(g, s) &= \prod_v Z_v(g_v, s) \\
	&= \Gamma_\R(s)^{r_1}\Gamma_\C(s)^{r_2}\prod_\id{p}\frac{1}{1 - N(\id{p})^{-s}} \\
	&= Z_K(s).
\end{align*} 
Had we started here, we could use some basic properties of zeta integrals to prove 
the functional equation of $Z_K$ mentioned previously. Obtaining the $s(s-1)$ factor of $\xi_K$ is more difficult. To do so, 
we will now restrict our attention to number fields with precisely one archimedean place (so that $K = \Q$, or $K$ is imaginary 
quadratic). Denote this single archimedean place by $\infty$, set $f_\id{p} = g_\id{p}$ as before, and let
$$
f_\infty(x) = 
\begin{dcases}
2\pi x^2(2\pi x^2 - 3)e^{-\pi x^2} \quad &\text{if $\infty$ is real, } \\
4\pi|x|^2(\pi|x|^2 - 1)e^{-2\pi |x|^2} &\text{if $\infty$ is complex.}
\end{dcases}
$$
\begin{lemma}
	With \(f_{\infty}\) defined above and $s \in \C^\times$, we have \(\widehat{f}_{\infty} = f_{\infty}\) and
	\[Z_{\infty}(f_{\infty}, s) = s(s-1)\Gamma_{K_\infty}(s)\]
	where $\Gamma_{K_\infty}$ is the Gamma factor for $K_\infty$.
\end{lemma}

\begin{proof}
First, suppose \(\infty\) is real and let \(g_1(x) = e^{-\pi x^2}\). Then 
\begin{align*}
f_{\infty}(x) &= 2\pi x^2(2\pi x^2 - 3)e^{-\pi x^2} \\
	&= \frac{d}{dx}\left[ x^2g_1'(x)\right].
\end{align*}
A straightforward computation shows that
\[2\int_0^\infty x^s g_1(x)\frac{dx}{x} = \Gamma_\R(s),\] and because \(g_1\) is a Schwartz function, it satisfies the rapidly decreasing condition: 
for any integers \(n, m \geq 0\), 
\[\lim_{x \to \infty} x^m \frac{d^n}{dx^n}g_1(x) = 0.\]
Hence, integration by parts gives
\begin{align*}
Z_{\infty}(f_{\infty}, s) &= 2\int_0^\infty x^s \frac{d}{dx}\left[ x^2g_1'(s)\right]\frac{dx}{x} \\
	&= 2s(s-1)\int_0^\infty x^sg_1(x)\frac{dx}{x} \\ 
	&= s(s-1)\Gamma_\R(s).
\end{align*}

\noindent Next suppose \(\infty\) is complex. Then 
\[Z_{\infty}(f_{\infty}, s) = \int_{\C^\times} f_{\infty}(z)|z|^{2s}d^\times z.\] 
To evaluate the above integral, do the change of variables \(z = re^{i\theta}\). 
Then \(d^\times x = \frac{2drd\theta}{r}\) and 
\begin{align*}
Z_{\infty}(f_{\infty}, s) &= \int_{\C^\times} f_{\infty}(z)|z|^{2s}d^\times z \\
	&= 4\pi \int_0^{2\pi}\int_0^\infty r^2(\pi r^2 - 1)e^{-2\pi r^2}r^{2s}\frac{2drd\theta}{r} \\
	&= 16\pi^2 \int_0^\infty r^{2s+1} (\pi r^2 - 1)e^{-2 \pi r^2} dr.
\end{align*}
Let \(g_2(r) = e^{-2\pi r^2}\) and note that 
\[\frac{d}{dr}\left[r^3g'_2(r)\right] = 16\pi r^3(\pi r^2 - 1)e^{-2\pi r^2}.\]
Hence, the integral we are evaluating becomes 
\[Z_{\infty}(f_{\infty}, s) = \pi\int_0^\infty r^{2s-2}\frac{d}{dr}\left[r^3g'_2(r)\right] dr.\]
 \(g_2\) satisfies the same rapidly decreasing condition as \(g_1\), so integration by parts gives us
\begin{align*}
Z_{\infty}(f_{\infty}, s) &= \pi(2s-2)2s\int_0^\infty r^{2s-1}g(r)dr \\
	&= 4\pi s(s-1)\int_0^\infty r^{2s-1}e^{-2\pi r^2}dr \\
	&= s(s-1)(2\pi)^{1-s}\Gamma(s) \\
	&= s(s-1)\Gamma_\C(s).
\end{align*}
To prove the self-duality of \(f_{\infty}\), let \(f = \bigotimes_v f_v\), and \(g\) the standard adelic function for which 
\(Z(g, s) = Z_K(s)\). Then 
\[Z(\widehat{f}, 1-s) = Z(f, s) = s(s-1)Z(g, s) = s(s-1)Z(\widehat{g}, 1-s).\]
Because \(f\) and \(g\) differ only in the \(\infty\) place, we must have 
\[Z(\widehat{f}_{\infty}, 1-s) = s(s-1)Z(\widehat{g}_{\infty}, 1-s) = s(s-1)Z(g_{\infty}, 1-s) = Z(f_{\infty}, 1-s).\]
By the uniqueness of the zeta integral, we must have \(\widehat{f}_{\infty} = f_{\infty}\).
\end{proof}

Hence, taking $f = \bigotimes_v f_v$, we have $Z(f, s) = s(s-1)Z_K(s)$ for every $s \in \C^\times$.

\subsection{A probabilistic interpretation of Dedekind zeta functions of three number fields}
Suppose now that $K$ is one of the fields $\Q, \Q(\sqrt{-1})$ or $\Q(\sqrt{-2})$. We will use 
our function $f:\A_K^\times \to \C$ defined above to construct a random variable $X$ so that 
$\E[X^s] = c_K^{-1}s(s-1)Z_K(s)$. 

For $t \geq 0$, let $\A_t^\times$ denote the set of ideles of absolute value $t$, so that 
\begin{align*}
	Z(f, s) &= \int_{\A^\times}|x|^sf(x)d^\times x \\
		&= \int_0^\infty t^s\left(\int_{\A^\times_t}f(x)d^\times x\right) \frac{dt}{t}.
\end{align*}
Our goal is now to think of the function $c_K^{-1}t^{-1}\int_{\A^\times_t} f(x)d^\times x$ as the density of some 
random variable. 

\begin{lemma}\thlabel{positivelem}
For each $K$ above and $t \geq 0$, 
$$c_K^{-1}t^{-1}\int_{\A^\times_t}f(x)d^\times x \geq 0.$$ 
\end{lemma}

\begin{proof}
Of course the factor $t^{-1}$ is nonnegative, and it is easy to check that 
$$
c_K = 
\begin{cases}
1 \quad &\text{if } K = \Q, \\
\frac{\pi}{2} &\text{if } K = \Q(\sqrt{-1}), \\
\pi &\text{if } K = \Q(\sqrt{-2}),
\end{cases}
$$
so that $c_K^{-1} > 0$; it now suffices to show that the integral is nonnegative as well. 
We can use the following trick hinted at before: 
$$\int_{\A^\times_t} f(x)d^\times x = \int_{\A^\times_t/K^\times}\left(\sum_{\kappa \in K^\times}f(x\kappa)\right)d^\times x$$
so it suffices to show that 
$$\sum_{\kappa \in K^\times}f(x\kappa) \geq 0$$
for all $x \in \A^\times$. 

First suppose that $K = \Q$ so that 
$$f_\infty(x_\infty) = 2\pi x_\infty^2(2\pi x_\infty^2-3)e^{-\pi x_\infty^2},$$
and for each finite prime $p$, 
$$f_p = 1_{\Z_p}.$$
We are considering 
the sum 
$$\sum_{q \in \Q^\times} f(xq).$$
For a given $x = (x_v)_v \in \A^\times$, we have that $f(xq) = 0$ if $|x_pq|_p > 1$ for just one 
prime $p$. However, $|x_p|_p \neq 1$ for only finitely many $p$, say $p \in S$. Letting 
\[r = \prod_{p \in S}p^{v_p(x_p)}\]
we have $|x_pq|_p = |rq|_p$ for every prime $p$. However, $0 < |xq|_p = |rq|_p \leq 1$ for every 
prime $p$ if and only if $rq \in \Z\setminus\{0\}$, so $q = n/r$ for $n \in \Z\setminus\{0\}$. 
Hence, 
$$\sum_{q \in \Q^\times} f(xq) = \sum_{n \in \Z\setminus\{0\}}f_\infty\left(x_\infty \frac{n}{r}\right).$$
Let $y = x_\infty/r$. If $|y| \geq 1$, then every term in the sum on the right hand side is positive. If 
$0 < |y| < 1$, then by the Poisson summation formula and the fact that $\widehat{f}_\infty = f_\infty$, 
$$\sum_{n \in \Z\setminus\{0\}}f_\infty(yn) = |y|^{-1}\sum_{n \in \Z\setminus\{0\}}f_\infty(y^{-1}n) \geq 0.$$
Either way the sum is nonnegative, so that 
$$\sum_{q \in \Q^\times}f(xq) \geq 0$$
for any $x \in \A^\times$. 

Now suppose that $K = \Q(\sqrt{-d})$, $d = 1, 2$, so that 
$\ocal_K = \Z\oplus \sqrt{-d}\Z$ and 
$$f_\infty(x_\infty) = 4\pi|x_\infty|^2(\pi|x_\infty|^2-1)e^{-2\pi|x_\infty|^2}.$$
For the same reasons as the rational case, we can reduce our sum to 
$$\sum_{\kappa \in K^\times}f(x\kappa) = \sum_{m \in M\setminus\{0\}}f_\infty(x_\infty m)$$
where $M$ is a fractional ideal of $\ocal_K$. Because both quadratic imaginary fields being discussed have 
class number $1$ and $x_\infty$ is arbitrary, we can suppose without loss of generality that our sum is of the form 
$$\sum_{\ell \in \ocal_K\setminus\{0\}}f_\infty(x_\infty \ell).$$
When $|x_\infty|^2 \geq 1/\pi$, this sum is clearly positive. Suppose then that $0 < |x_\infty|^2 < 1/\pi$ 
and let $u = x_\infty^{-1}$, so by Poisson summation and that $\widehat{f}_\infty = f_\infty$
$$\sum_{\ell \in \ocal_K\setminus\{0\}}f_\infty(x_\infty \ell) = |u|\sum_{\ell^* \in \ocal_K^*\setminus\{0\}}f_\infty(u\ell^*)$$
where $\ocal_K^* = \frac{1}{2\sqrt{-d}}(\Z \oplus \sqrt{-d}\Z)$. But $|u|^2 \geq \pi$, and 
$|\ell^*|^2 \geq \frac{1}{4d}$, so when $d = 1, 2$ we have $\pi|u\ell^*|^2 -1 \geq \frac{\pi^2}{4d}-1 > 0$, 
so that each term of the sum is positive and therefore
$$\sum_{\ell \in \ocal_K\setminus\{0\}}f_\infty(x_\infty \ell) = |u|\sum_{\ell^* \in \ocal_K^*\setminus\{0\}}f_\infty(u\ell^*) \geq 0.$$

Hence for each $K$ above and $t \geq 0$, 
$$c_K^{-1}t^{-1}\int_{\A^\times_t}f(x)d^\times x = 
c_K^{-1}t^{-1}\int_{\A^\times_t/K^\times}\left(\sum_{\kappa \in K^\times}f(x\kappa)\right)d^\times x \geq 0.$$
\end{proof}

\begin{remark}
For imaginary quadratic fields of discriminant $d\geq3$ and class number $1$, 
we no longer have that the sum $\sum_{\ell^*\in \ocal_K\setminus\{0\}}f(u\ell^*)$ 
has only positive summands, so it is unclear if the sum would be positive. 
Fields which are neither the rationals nor imaginary quadratic have more than one archimedean place, and again 
one runs into issues applying the same Poisson summation argument directly, so a stronger argument would be required to generalize to 
other number fields. 
\end{remark}

We can now prove our main theorem: 

\begin{proof}[Proof of \thref{rv}]
For $t > 0$, and $f$ defined above, let 
$$\psi(t) := c_K^{-1}t^{-1}\int_{\A_t^\times}f(x)d^\times x.$$
Then $\psi(t) \geq 0$ by \thref{positivelem}, for each $s \in \C^\times$ 
$$\int_0^\infty t^s\psi(t)dt = c_K^{-1}\int_0^\infty t^s \left(\int_{\A^\times_t}f(x)d^\times x\right)\frac{dt}{t} = c_K^{-1}s(s-1)Z_K(s),$$
and as $s \to 0$, $c_K^{-1}s(s-1)Z_K(s) \to 1$, so that $\int_0^\infty \psi(t)dt = 1$. 
Hence, $\psi$ is the density of a random variable satisfying the above conditions. 
\end{proof}

\begin{remark}
We can view the random variable $X$ from a more adelic perspective. Note that in proving \thref{positivelem} and \thref{rv} we have 
additionally shown that $c_K^{-1}\left(\sum_{\kappa \in K^\times} f(x\kappa)\right)d^\times x$ is 
a probability measure on the idele class group $\A^\times/K^\times$. Letting $Y$ be the idelic random variable with this distribution, and $|\cdot|$ 
the idele-norm, 
we then have that $\E |Y|^s = |D|^{-s/2}\xi_K(s) = \E(X^s)$ for every $s \in \C$. Hence $\log X$ and $\log|Y|$ have the same moment generating functions, 
so that $X = |Y|$ almost surely. 
\end{remark}

\subsection{An application to Li coefficients}

The \(n\)th Li coefficient of \(\xi_K\) is defined in \cite{li} as
\[\lambda_n := \frac{1}{(n-1)!}\frac{d^n}{ds^n}\left[ s^{n-1}\log \xi_K(s)\right]_{s = 1}.\]
Using the probabilistic interpretation of \(\xi_K\), we prove the following: 

\begin{proposition}
Let \(K\) be one of $\Q$, $\Q(\sqrt{-1})$, or $\Q(\sqrt{-2})$, and $\lambda_n$ the Li coefficients for $\xi_K$. Then 
\(\lambda_1, \lambda_2 > 0\). 
\end{proposition}

\begin{proof}
Let \(X\) be the random variable constructed in \thref{rv} so that for each \(s \in \C\), 
\[\E(X^s) = |D|^{-s/2}\xi_K(s),\]
and functional equation of \(\xi_K(s)\) gives 
\[\E(X^{1-s}) = |D|^{(s-1)/2}\xi_K(s).\]
Let \(L := -\log X\), so that the first cumulant of \(L\) is given by 
\begin{align*}
\kappa_1 &= \frac{d}{ds}\left[ \log \E(e^{(s-1)L}) \right]_{s = 1} \\
	&= \frac{d}{ds} \left[ \log \E (X^{1-s}) \right]_{s=1} \\
	&= \frac{d}{ds}\left[ \log\left( |D|^{(s-1)/2}\xi_K(s)\right)\right]_{s=1} \\
	&= \log\sqrt{|D|} + \lambda_1.
\end{align*}
But by a basic property of cumulants and Jensen's inequality, 
\begin{align*}
\kappa_1 = \E(L) = \E(-\log X) > -\log \E(X) = \log\sqrt{|D|}
\end{align*}
so \(\lambda_1 > 0\). Similarly, the second cumulant of \(L\) is 
\begin{align*}
    \kappa_2 &= \frac{d^2}{ds^2}\log\left[|D|^{(s-1)/2}\xi_K(s)\right]_{s=1} \\
    &= \frac{d^2}{ds^2}\left[\log\xi_K(s)\right]_{s=1}
\end{align*}
and \(\kappa_2 = \text{Var} (L) \geq 0\), so that
\begin{align*}
    \lambda_2 &= \frac{d^2}{ds^2}\left[s\log\xi_K(s) \right]_{s=1} \\
    &= \frac{d}{ds}\left[\log\xi_K(s) + s\frac{d}{ds}\log\xi_K(s)\right]_{s=1} \\
    &= 2\lambda_1 + \left[s\frac{d^2}{ds^2}\log\xi_K(s)\right]_{s=1} \\
    &= 2\lambda_1 + \kappa_2 \\
    &> 0.
\end{align*}
Therefore, the first two Li coefficients are positive for $\Q$, $\Q(\sqrt{-1})$, and $\Q(\sqrt{-2})$.
\end{proof}

\subsection*{Acknowledgements}
\noindent I am incredibly grateful to Professor M. Ram Murty for his supervision during this project. 
	His comments, encouragement, and advice had a profound impact on both the outcome of this paper and my mathematical development. 
	I would also like to thank Professor Francesco Cellarosi and Professor Daniel Johnstone for their helpful comments and aid while working on this paper.

\end{document}